\documentclass{amsart}
\usepackage{amssymb}
\usepackage{amsmath}
\usepackage{amsfonts}

\theoremstyle{plain}

\newtheorem{corollary}{Corollary}

\newtheorem{definition}{Definition}

\newtheorem{proposition}{Proposition}

\numberwithin{equation}{section}
\input{tcilatex}

\begin{document}
\title[Clifford Analysis on Orlicz-Sobolev Spaces]{Clifford Analysis on
Orlicz-Sobolev Spaces}
\author{Dejenie Alemayehu Lakew}
\address{John Tyler Community College\\
Department of Mathematics\\
USA}
\email{dlakew@jtcc.edu}
\urladdr{http://www.jtcc.edu}
\author{Mulugeta Alemayehu Dagnaw}
\curraddr{Debre Tabor University\\
Department of Mathematics\\
Ethiopia}
\email{malemayehu3@gmail.com}
\urladdr{http://www.dtu.edu.et}
\date{September 13, 2014}
\dedicatory{}
\thanks{This paper is in final form and no version of it will be submitted
for publication elsewhere.}

\begin{abstract}
In this article we develop few of the analogous theoretical results of
Clifford analysis over Orlicz-Sobolev spaces and study mapping properties of
the Dirac operator $D=\sum_{j=1}^{n}e_{j}\partial _{x_{j}}$ and the
Teodorescu transform $\tau _{\Omega }$ over these function spaces. We also
get analogous decomposition results $\tciLaplace ^{\psi }\left( \Omega
,Cl_{n}\right) =A^{\psi }\left( \Omega ,Cl_{n}\right) \dotplus \overline{D}%
\left( W_{0}^{1,\psi }\left( \Omega ,Cl_{n}\right) \right) $ of Clifford
valued Orlicz spaces and the generalized Orlicz - Sobolev spaces $W^{k,\psi
}\left( \Omega ,Cl_{n}\right) =A^{k,\psi }\left( \Omega ,Cl_{n}\right)
\dotplus \overline{D}\left( W_{0}^{k+1,\psi }\left( \Omega ,Cl_{n}\right)
\right) $ where $\psi $ is an Orlicz function and $k\in 
\mathbb{N}
\cup \{0\}$.
\end{abstract}

\maketitle

\section{$\mathbf{Introduction}$}

\textit{Clifford} analysis is a theoretical study of \textit{Clifford}
valued functions that are null solutions to the \textit{Dirac} or \textit{%
Dirac like} differential operators and their applications over the regular
continuous function spaces $C^{k}\left( \Omega ,Cl_{n}\right) $, \textit{%
Lipschitz} spaces $C^{k,\lambda }\left( \Omega ,Cl_{n}\right) $ and over 
\textit{Sobolev} and \textit{Slobodeckji} spaces $W^{k,p}\left( \Omega
,Cl_{n}\right) ,$ $W^{k+\lambda ,p}\left( \Omega ,Cl_{n}\right) $
respectively for $0<\lambda <1$. The latter spaces are the right viable
search spaces for solutions to most partial differential equations where we
seek functions that are weakly differentiable as regular functions are
scarce. All available literatures are done over function spaces I have
indicated and the domain $\Omega $ in most cases is a bounded or unbounded
but smooth region in \textit{Euclidean} spaces $%
\mathbb{R}
^{n}$ or a manifold in $%
\mathbb{R}
^{n}$ or domain manifold in $%
\mathbb{C}
^{n}$ with being \textit{Lipschitz}, the minimally smoothness condition. In
this paper we look at some analogous results of \textit{Clifford} analysis
over $Cl_{n}$- valued Orlicz and Orlicz - \textit{Sobolev} spaces such as $%
\tciLaplace ^{\psi }\left( \Omega ,Cl_{n}\right) $ and $W^{k,\psi }\left(
\Omega ,Cl_{n}\right) $ where $\psi $ is an \textit{Orlicz} or \textit{Young}
function.

\bigskip\ Let $\{e_{j}:j=1,2,...,n\}$ be an orthonormal basis for $%
\mathbb{R}
^{n}$ that is equipped with an inner product so that 
\begin{equation}
e_{i}e_{j}+e_{j}e_{i}=-2\delta _{ij}e_{0}  \label{inner product 1}
\end{equation}%
where $\delta _{ij}$ is the \textit{Kronecker delta}. The \textit{inner
product} defined satisfies an anti commutative relation 
\begin{equation}
x^{2}=-\Vert x\Vert ^{2}  \label{inner product 2}
\end{equation}

and with this\textit{\ inner product, } $%
\mathbb{R}
^{n}$ generates a $2^{n}$-dimensional non commutative algebra called \textit{%
Clifford} algebra which is denoted by $Cl_{n}$.

\ 

\textbf{Basis for }$Cl_{n}$ : \ The family 
\begin{equation*}
\{e_{A}:A\subset \{1<2<3<...<n\}\}
\end{equation*}

is a basis for the algebra. The object $e_{0}$ used above is the identity
element of the \textit{Clifford} algebra $Cl_{n}$. \ 

\ 

\textbf{Representation of elements of }$Cl_{n}$: Every element $a\in Cl_{n}$
is represented by%
\begin{equation}
a=\sum e_{A}a_{A}  \label{Clifford element}
\end{equation}%
where $a_{A}$ is a real number for each $A$.

\ 

\textbf{Embedding} : By identifying $x=(x_{1},x_{2},...,x_{n})\in 
\mathbb{R}
^{n}$ with $\sum_{j=1}^{n}e_{j}x_{j}$ of $Cl_{n}$ we have an embedding 
\begin{equation*}
\mathbb{R}
^{n}\hookrightarrow Cl_{n}
\end{equation*}

\textbf{Clifford conjugation:} $\ \overline{a}$ of a \textit{Clifford}
element $a=\sum e_{A}a_{A}\in Cl_{n}$ is defined as: 
\begin{equation*}
\overline{a}=\sum \overline{e}_{A}a_{A}
\end{equation*}%
where%
\begin{equation*}
\overline{e}_{A}=\overline{e_{j_{1}}...e}_{j_{r}}=\left( -1\right)
^{r}e_{j_{r}}...e_{j_{1}}
\end{equation*}

with particulars:

\begin{equation*}
\overline{e}_{j}=-e_{j},\text{ \ }e_{j}^{2}=-1
\end{equation*}
for $i,j=1,2,...,n$ and for 
\begin{equation*}
i\neq j:\overline{e_{i}e}_{j}=(-1)^{2}e_{j}e_{i}=e_{j}e_{i}
\end{equation*}

\begin{definition}
(Clifford norm) \bigskip For $a=\sum e_{A}a_{A}\in Cl_{n}$ we define the
Clifford norm of \ $a$ by 
\begin{equation}
\Vert a\Vert _{Cl_{n}}=\left( \left( a\overline{a}\right) _{0}\right) ^{%
\frac{1}{2}}=\left( \underset{A}{\sum a_{A}^{2}}\right) ^{\frac{1}{2}}
\label{Clifford Norm}
\end{equation}%
where $\left( a\right) _{0}$ is the real part of $a\overline{a}$.
\end{definition}

\ \ \ \ \ \ \ \ \ \ \ \ \ \ \ 

The Clifford norm $\Vert .\Vert _{Cl_{n}}$ satisfies the inequality: 
\begin{equation}
\Vert ab\Vert _{Cl_{n}}\leq c\left( n\right) \Vert a\Vert _{Cl_{n}}\Vert
b\Vert _{Cl_{n}}  \label{Norm Inequality}
\end{equation}%
with $c\left( n\right) $ a dimensional constant.

\ 

\textbf{Kelvin inversion:} Each non zero element $x\in 
\mathbb{R}
^{n}$ has an inverse given by :%
\begin{equation}
x^{-1}=\frac{\overline{x}}{\Vert x\Vert _{Cl_{n}}^{2}}  \label{inverse}
\end{equation}

$\sphericalangle $ \ In this paper $\Omega $ is a bounded and smooth domain
of $%
\mathbb{R}
^{n}$ with at least a $C^{1}$ - hypersurface boundary.

\ \ \ \ 

\textbf{Function representation:} \bigskip A $Cl_{n}$- valued function $%
f:\Omega \longrightarrow Cl_{n}$ has a representation:$\ $

\begin{equation}
\ f=\sum_{A}e_{A}f_{A}  \label{Clifford valued function}
\end{equation}%
where $f_{A}:\Omega \longrightarrow 
\mathbb{R}
$ is a real valued component or section of $f$.

\ \ 

\begin{definition}
Let $f\in C^{1}\left( \Omega \right) \cap C\left( \overline{\Omega }\right) $%
, we define the Dirac\textbf{\ }derivative of $f$ by 
\begin{equation}
Df\left( x\right) =\sum_{j=1}^{n}e_{j}\partial _{x_{j}}f\left( x\right)
\label{Dirac}
\end{equation}

A function $f:\Omega \longrightarrow Cl_{n}$ is called left\textbf{\ }%
monogenic or left\textbf{\ }Clifford\textbf{\ }analytic over $\Omega $ if 
\begin{equation*}
Df\left( x\right) =0,\text{ }\forall x\in \Omega
\end{equation*}%
and likewise it is called right\textbf{\ }monogenic over $\Omega $ if 
\begin{equation*}
f(x)D=\sum_{j=1}^{n}\partial _{x_{j}}f\left( x\right) e_{j}=0,\text{ }%
\forall x\in \Omega
\end{equation*}
\end{definition}

An example of both left and right \textit{monogenic} function defined over $%
\mathbb{R}
^{n}\backslash \{0\}$ is given by 
\begin{equation*}
\Phi \left( x\right) =\frac{\overline{x}}{\omega _{n}\Vert x\Vert
_{Cl_{n}}^{n}}
\end{equation*}%
where $\omega _{n}=\frac{2\pi ^{\frac{n}{2}}}{\Gamma \left( \frac{n}{2}%
\right) }$ is the surface area of the unit sphere in $%
\mathbb{R}
^{n}$.

\ \ 

The function $\Phi $\ is also a fundamental solution to the Dirac operator $%
D $ and we define integral transforms as convolutions of $\Phi $ with
functions of some spaces below.

\ 

\begin{definition}
Let $f\in C^{1}\left( \Omega ,Cl_{n}\right) \cap C\left( \overline{\Omega }%
,Cl_{n}\right) $. We define two integral transforms as follow:

\begin{equation}
\text{Teodorescu or Cauchy transform }\text{: \ \ \ \ \ \ \ }\text{\ \ }%
\zeta _{\Omega }f\left( x\right) =\int_{\Omega }\Phi \left( y-x\right)
f\left( y\right) d\Omega _{y}=\left( \Phi \ast f\right) \left( x\right) 
\text{ ,\ }x\in \Omega
\end{equation}

\begin{equation}
\text{Feuter transform }\text{: \ \ \ \ \ \ \ \ }\xi _{\partial \Omega
}f\left( x\right) =\int_{\partial \Omega }\Phi \left( y-x\right) \upsilon
\left( y\right) f\left( y\right) d\partial \Omega _{y}=\left( \Phi \ast
\upsilon f\right) \left( x\right) \text{ , \ }x\notin \partial \Omega
\end{equation}

where $\upsilon \left( y\right) $ is a unit normal vector pointing outward
at $y\in \partial \Omega $ and "$\ast $" is a convolution.
\end{definition}

\ \ \ \ \ \ \ \ \ \ \ \ \ \ \ \ \ \ \ 

These transforms will also be extended to hold over Sobolev spaces $%
W^{k,p}\left( \Omega ,Cl_{n}\right) $ by continuity and denseness arguments.

\section{$\mathbf{Cl}_{n}\mathbf{-Valued\ Orlicz\ and\
Orlicz-Sobolev-Slobodeckji\ Spaces}$}

The function spaces we use in this paper are Cifford algebra valued Orlicz-
Sobolev - Slobodeckji spaces. We therefore start with the definition of
these spaces.

\ 

\begin{definition}
A function $\psi :[0,\infty )\longrightarrow \lbrack 0,\infty )$ is said to
be an Orlicz function if $\psi \left( 0\right) =0,\underset{x\longrightarrow
\infty }{\lim }\psi \left( x\right) =\infty $ and $\psi \nearrow $ and
convex on its domain.

An example of such a function is :$\psi \left( x\right) =\mid x\mid ^{2}$
and $\psi \left( x\right) =\mid x\mid ^{p}$ for $1<p<\infty $.
\end{definition}

\begin{definition}
Let $\psi :[0,\infty )\longrightarrow \lbrack 0,\infty )$ be an Orlicz
function. A measurable, locally integrable function $f\in \tciLaplace _{%
\text{loc}}\left( \Omega ,%
\mathbb{R}
\right) $ is said to belong to the Orlicz space $\tciLaplace ^{\psi }\left(
\Omega ,%
\mathbb{R}
\right) $ if 
\begin{equation*}
\exists \beta >0:\int_{\Omega }\psi \left( \frac{\mid f(x)\mid }{\beta }%
\right) d\Omega _{x}<\infty
\end{equation*}

We thus define the Orlicz space\textbf{\ }$\tciLaplace ^{\psi }\left( \Omega
,%
\mathbb{R}
\right) $ as\textbf{\ } 
\begin{equation*}
\tciLaplace ^{\psi }\left( \Omega ,%
\mathbb{R}
\right) =\{f\in \tciLaplace _{\text{loc}}\left( \Omega ,%
\mathbb{R}
\right) :\exists \beta >0:\int_{\Omega }\psi \left( \frac{\mid f(x)\mid }{%
\beta }\right) d\Omega _{x}<\infty \}
\end{equation*}%
with a norm called Luxembourg norm defined as : 
\begin{equation}
\Vert f\Vert _{\tciLaplace ^{\psi }\left( \Omega ,%
\mathbb{R}
\right) }=\inf \{\beta >0:\int_{\Omega }\psi \left( \frac{\mid f(x)\mid }{%
\beta }\right) d\Omega _{x}\leq 1\}  \label{luxemberg norm1}
\end{equation}
\end{definition}

\bigskip\ \ \ \ 

The \textit{Orlicz} power functions $\psi \left( x\right) =\mid x\mid ^{p}$
for $1<p<\infty $ provide the usual \textit{Lebesgue spaces} $\tciLaplace
^{p}\left( \Omega \right) $.

\ 

The theme here is to work Clifford analysis over such function spaces and
develop analogous results we have on the usual regular, Lebesgue and Sobolev
spaces. We start by defining how Clifford valued functions be in Orlicz
spaces.

\ 

\begin{definition}
A $Cl_{n}$-valued measurable and locally integrable function $%
f=\dsum\limits_{A}e_{A}f_{A}$ over $\Omega $ is said to be in the Orlicz
space 
\begin{equation*}
f\in \tciLaplace ^{\psi }\left( \Omega ,Cl_{n}\right) \Leftrightarrow
f_{A}\in \tciLaplace ^{\psi }\left( \Omega ,%
\mathbb{R}
\right)
\end{equation*}

with \textit{Clifford-Luxembourg norm}: 
\begin{equation}
\Vert f\Vert _{\tciLaplace ^{\psi }\left( \Omega ,Cl_{n}\right)
}=\dsum\limits_{A}\Vert f_{A}\Vert _{\tciLaplace ^{\psi }\left( \Omega ,%
\mathbb{R}
\right) }  \label{Clifford-luxemberg norm}
\end{equation}
\end{definition}

The \textit{Clifford-Luxembourg norm} of $f$ is defined interns of the 
\textit{Luxembourg norm} of component real valued functions $f_{A}$.

\ \ 

We next define the $Cl_{n}-$ valued Orlicz-Sobolev spaces.

\ 

\begin{definition}
Let $\psi $ be an Orlicz function and $k\in 
\mathbb{N}
\cup \{0\}$. We define the Orlicz-Sobolev space $W^{k,\psi }\left( \Omega
,Cl_{n}\right) $ as 
\begin{equation*}
W^{k,\psi }\left( \Omega ,Cl_{n}\right) =\{f\in \tciLaplace _{loc}\left(
\Omega ,Cl_{n}\right) :\left( \forall A\right) \left( \exists \beta
_{A}>0\right) :\dsum\limits_{0\leq \mid \alpha \mid \leq k}\int_{\Omega
}\psi \left( \frac{\mid D^{\alpha }f_{A}\left( x\right) \mid }{\beta _{A}}%
\right) d\Omega _{x}<\infty \}
\end{equation*}

with norm (Clifford-Luxembourg ) 
\begin{equation}
\parallel f\parallel _{W^{k,\psi }\left( \Omega ,Cl_{n}\right)
}=\dsum\limits_{A}\dsum\limits_{0\leq \mid \alpha \mid \leq k}\parallel
f_{A}\parallel _{_{\tciLaplace ^{\psi }\left( \Omega ,%
\mathbb{R}
\right) }}=\dsum\limits_{A}\parallel f_{A}\parallel _{W^{k,\psi }\left(
\Omega ,%
\mathbb{R}
\right) }
\end{equation}

where $\ $%
\begin{equation*}
\parallel f_{A}\parallel _{W^{k,\psi }\left( \Omega ,%
\mathbb{R}
\right) }:=\inf \{\beta _{A}>0:\dsum\limits_{0\leq \mid \alpha \mid \leq
k}\int_{\Omega }\psi \left( \frac{\mid D^{\alpha }f_{A}\left( x\right) \mid 
}{\beta _{A}}\right) d\Omega _{x}\leq 1\}\text{ \ \ }
\end{equation*}

When $k=0$ we have $\tciLaplace ^{\psi }\left( \Omega ,Cl_{n}\right) $ and 
\begin{equation*}
f\in \tciLaplace ^{\psi }\left( \Omega ,Cl_{n}\right) \Longleftrightarrow
f_{A}\in \tciLaplace ^{\psi }\left( \Omega ,%
\mathbb{R}
\right)
\end{equation*}%
with 
\begin{equation}
\parallel f\parallel _{\tciLaplace ^{\psi }\left( \Omega ,Cl_{n}\right)
}=\dsum\limits_{A}\inf \{\lambda _{A}>0:\int_{\Omega }\psi \left( \frac{\mid
f_{A}(x)\mid }{\lambda _{A}}\right) d\Omega _{x}\leq
1\}=\dsum\limits_{A}\parallel f_{A}\parallel _{\tciLaplace ^{\psi }\left(
\Omega ,%
\mathbb{R}
\right) }  \label{Clifford-luxemberg norm2}
\end{equation}

We also define traceless Sobolev spaces as

\begin{equation*}
W_{0}^{k,\psi }\left( \Omega ,Cl_{n}\right) :=\{f\in W^{k,\psi }\left(
\Omega ,Cl_{n}\right) :f_{\mid \partial \Omega }=\underset{1\leq \left\vert
\alpha \right\vert \leq k-1}{D^{\alpha }f_{\mid \partial \Omega }}=0\}
\end{equation*}
\end{definition}

The generalized Orlicz-Slobodeckji spaces are defined as

\begin{definition}
The Orlicz - Slobodeckji spaces 
\begin{equation*}
\widetilde{W}^{k-1,\psi ,\psi }\left( \partial \Omega ,Cl_{n}\right)
:=\{g=\tau f:f\in W^{k,\psi }\left( \Omega ,Cl_{n}\right) \}
\end{equation*}
with associated norm :

\begin{eqnarray*}
\Vert g\Vert _{\widetilde{W}^{k-1,\psi ,\psi }\left( \partial \Omega
,Cl_{n}\right) } &=&\underset{\Vert \alpha \Vert \leq k-1}{\sum }%
\int_{\partial \Omega }\psi \left( \frac{\left\vert \left( D^{\alpha
}g|\right) \right\vert }{\lambda }\right) d\partial \Omega _{x} \\
&&+\underset{\Vert \alpha \Vert =k-1}{\sum }\dint\limits_{\partial \Omega
}\dint\limits_{\partial \Omega }\psi \left( \frac{|D^{\alpha }g\left(
x\right) -D^{\alpha }g\left( y\right) |}{\lambda |x-y|}\right) \left\vert
x-y\right\vert ^{2-n}d\partial \Omega _{x}d\partial \Omega _{y}
\end{eqnarray*}

when $k=1$, we have 
\begin{equation*}
\widetilde{W}^{0,\psi ,\psi }\left( \partial \Omega ,Cl_{n}\right)
=\tciLaplace ^{\psi ,\psi }\left( \partial \Omega ,Cl_{n}\right)
\end{equation*}

These Orlicz-Slobodeckji spaces are analogues of the Sobolev-Slobodeckji
spaces

\begin{equation*}
W^{k-\frac{1}{p},p}\left( \partial \Omega ,Cl_{n}\right) :=\{g=\tau f:f\in
W^{k,p}\left( \Omega ,Cl_{n}\right) \}
\end{equation*}%
for $k\in 
\mathbb{N}
$.
\end{definition}

\begin{proposition}
The Slobodeckji space $W^{1-\frac{1}{p},p}\left( \partial \Omega \right) $
with $\lambda =1-\frac{1}{p}$ so that $\lfloor \lambda \rfloor =0$ and $%
\{\lambda \}=1-\frac{1}{p}$ and for $f\in W^{1-\frac{1}{p},p}\left( \partial
\Omega \right) $ we have 
\begin{equation*}
\parallel f\parallel _{W^{1-\frac{1}{p},p}\left( \partial \Omega \right) }=%
\text{ }\parallel f\parallel _{\tciLaplace ^{p}\left( \partial \Omega
\right) }+\left( \dint\limits_{\partial \Omega }\dint\limits_{\partial
\Omega }\left( \frac{|f\left( x\right) -f\left( y\right) |}{\left\vert
x-y\right\vert }\right) ^{p}\left\vert x-y\right\vert ^{2-n}d\partial \Omega
_{x}d\partial \Omega _{y}\right) ^{\frac{1}{p}}
\end{equation*}
\end{proposition}

\begin{proof}
The proof is short and straight forward by considering $\{\lambda \}=1-\frac{%
1}{p}$, $[\lambda ]=0$ so that the singularity exponent of the integrand
wiil be%
\begin{eqnarray*}
\left\vert x-y\right\vert ^{-\left( \dim (\partial \Omega )+\{\lambda
\}p\right) } &=&\left\vert x-y\right\vert ^{-\left( n-1+\{\lambda \}p\right)
} \\
&=&\left\vert x-y\right\vert ^{-\left( n-1+\left( 1-\frac{1}{p}\right)
p\right) } \\
&=&\left\vert x-y\right\vert ^{-\left( n-1+p-1\right) } \\
&=&\frac{\left\vert x-y\right\vert ^{2-n}}{\left\vert x-y\right\vert ^{p}}
\end{eqnarray*}

which provides the factor expression of the integrand of the right term of
the right hand side of the two summands of the norm and the actual norm
follows form the definition of norm of Slobodeckji space $W^{\lambda
,p}\left( \partial \Omega \right) $.

\ 
\end{proof}

\begin{proposition}
The Orlicz-Slobodeckji space $\tciLaplace ^{\psi ,\psi }\left( \partial
\Omega ,Cl_{n}\right) $ has the following norm: for $f\in \tciLaplace ^{\psi
,\psi }\left( \partial \Omega ,Cl_{n}\right) $, 
\begin{equation*}
\parallel f\parallel _{\tciLaplace ^{\psi ,\psi }\left( \partial \Omega
\right) }=\parallel f\parallel _{\tciLaplace ^{\psi }\left( \partial \Omega
,Cl_{n}\right) }+\dint\limits_{\partial \Omega }\dint\limits_{\partial
\Omega }\psi \left( \frac{\left\vert f(x)-f(y)\right\vert }{\lambda
\left\vert x-y\right\vert }\right) \left\vert x-y\right\vert ^{2-n}d\partial
\Omega _{x}d\partial \Omega _{y}
\end{equation*}

with $\lambda >0$.
\end{proposition}

\section{$\mathbf{Mapping\ Properties\ of\ D,\mathbf{\protect\zeta }_{\Omega
}\ and\ \protect\zeta }_{\partial \Omega }$}

The three operators, the Dirac operator $D$, the Teodorescu or Cauchy
transform $\zeta _{\Omega }$ and the Feuter transform $\ \mathbf{\zeta }%
_{\partial \Omega }\mathbf{\ }$keep integrability invariant but change
regularity (smoothness) over Sobolev spaces in the following ways:

\ 

\begin{proposition}
The Dirac operator $D:W^{k,\psi }\left( \Omega ,Cl_{n}\right)
\longrightarrow W^{k-1,\psi }\left( \Omega ,Cl_{n}\right) $ with

\begin{equation*}
\Vert Df\Vert _{W^{k-1,\psi }\left( \Omega ,Cl_{n}\right) }\leq \gamma \Vert
f\Vert _{W^{k,\psi }\left( \Omega ,Cl_{n}\right) }
\end{equation*}%
for $\gamma =\gamma \left( n,\psi ,\Omega \right) $ a positive constant.
\end{proposition}

\ 

\begin{proof}
Let $f\in W^{k,\psi }\left( \Omega ,Cl_{n}\right) $. We need to show that

\begin{eqnarray*}
\Vert Df\Vert _{W^{k-1,\psi }\left( \Omega ,Cl_{n}\right) }
&=&\dsum\limits_{0\leq \mid \alpha \mid \leq k-1}\parallel D^{\alpha }\left(
Df\right) \parallel _{\tciLaplace ^{\psi }\left( \Omega ,Cl_{n}\right) } \\
&=&\dsum\limits_{0\leq \mid \beta \mid \leq k}\parallel D^{\beta }f\parallel
_{\tciLaplace ^{\psi }\left( \Omega ,Cl_{n}\right) } \\
&\leq &\gamma \parallel f\parallel _{W^{k,\psi }\left( \Omega ,Cl_{n}\right)
}
\end{eqnarray*}
\end{proof}

\ \ \ \ \ \ \ \ \ \ \ \ \ \ \ \ \ \ \ 

\begin{proposition}
$D:\tciLaplace ^{\psi }\left( \Omega \right) \longrightarrow W^{-1,\psi
}\left( \Omega \right) $ where $\psi $ is an Orlicz function.
\end{proposition}

\ 

\begin{proof}
Let $f\in \tciLaplace ^{\psi }\left( \Omega ,Cl_{n}\right) $. Then

\begin{equation*}
\Vert Df\Vert _{W^{-1,\psi }\left( \Omega \right) }=\sup \{\frac{|\langle
Df,g\rangle |}{\Vert g\Vert _{W_{0}^{1,\psi \ast }\left( \Omega \right) }}%
:g\neq 0,g\in W_{0}^{1,\psi \ast }\left( \Omega \right) \}
\end{equation*}%
for $\psi $ and $\psi \ast $ are conjugate Orlicz functions.

\ 

But 
\begin{eqnarray*}
|\langle Df,g\rangle | &=&|\langle f,Dg\rangle |\leq \Vert f\Vert
_{\tciLaplace ^{\psi }\left( \Omega \right) }\Vert Dg\Vert _{\tciLaplace
^{\psi \ast }\left( \Omega \right) } \\
&\leq &\Vert f\Vert _{\tciLaplace ^{\psi }\left( \Omega \right) }\Vert
g\Vert _{W_{0}^{1,\psi \ast }\left( \Omega \right) }
\end{eqnarray*}

Thus by the Cauchy-Schwartz inequality we have

\begin{equation*}
\frac{|\langle Df,g\rangle |}{\Vert g\Vert _{W_{0}^{1,\psi \ast }\left(
\Omega \right) }}\leq \frac{\Vert f\Vert _{\tciLaplace ^{\psi }\left( \Omega
\right) }\Vert g\Vert _{W_{0}^{1,\psi \ast }\left( \Omega \right) }}{\Vert
g\Vert _{W_{0}^{1,\psi \ast }\left( \Omega \right) }}=\Vert f\Vert
_{\tciLaplace ^{\psi }\left( \Omega \right) }
\end{equation*}%
Therefore

\begin{eqnarray*}
\Vert Df\Vert _{W^{-1,\psi }\left( \Omega \right) } &=&\sup \{\frac{|\langle
Df,g\rangle |}{\Vert g\Vert _{W_{0}^{1,\psi \ast }\left( \Omega \right) }}%
:g\neq 0,\text{ }g\in W_{0}^{1,\psi \ast }\left( \Omega \right) \} \\
&\leq &\sup \{\frac{\Vert f\Vert _{\tciLaplace ^{\psi }\left( \Omega \right)
}\Vert g\Vert _{W_{0}^{1,\psi \ast }\left( \Omega \right) }}{\Vert g\Vert
_{W_{0}^{1,\psi \ast }\left( \Omega \right) }}:g\neq 0,\text{ }g\in
W_{0}^{1,\psi \ast }\left( \Omega \right) \} \\
&=&\Vert f\Vert _{\tciLaplace ^{\psi }\left( \Omega \right) }
\end{eqnarray*}

\ \ 
\end{proof}

\ \ \ \ \ \ \ \ \ \ \ \ \ \ \ \ 

\begin{proposition}
Let $k\in 
\mathbb{N}
\cup \{0\}$ and $\psi $ be an Orlicz function. Then there exists a positive
constant $\beta =\beta \left( n,\psi ,\Omega \right) $ such that

\ 
\begin{equation}
\zeta _{\Omega }:W^{k,\psi }\left( \Omega ,Cl_{n}\right) \longrightarrow
W^{k+1,\psi }\left( \Omega ,Cl_{n}\right)  \label{Theodorescu property}
\end{equation}%
with%
\begin{equation*}
\Vert \zeta _{\Omega }f\Vert _{W^{k+1,\psi }\left( \Omega ,Cl_{n}\right)
}\leq \beta \Vert f\Vert _{W^{k,\psi }\left( \Omega ,Cl_{n}\right) }
\end{equation*}
\end{proposition}

\begin{proof}
Let $f\in W^{k,\psi }\left( \Omega ,Cl_{n}\right) $. Then clearly $\zeta
_{\Omega }f\in W^{k+1,\psi }\left( \Omega ,Cl_{n}\right) $ as $D\zeta
_{\Omega }f=f$ from Borel-Pompeiu relation and we have norm estimates 
\begin{eqnarray*}
\Vert \zeta _{\Omega }f\Vert _{W^{k+1,\psi }\left( \Omega ,Cl_{n}\right) }
&=&\dsum\limits_{0\leq \left\vert \alpha \right\vert \leq k+1}\parallel
D^{\alpha }\zeta _{\Omega }f\parallel _{\tciLaplace ^{\psi }\left( \Omega
,Cl_{n}\right) } \\
&=&\dsum\limits_{0\leq \left\vert \beta \right\vert \leq k}\parallel
D^{\beta }(D\zeta _{\Omega }f)\parallel _{\tciLaplace ^{\psi }\left( \Omega
,Cl_{n}\right) } \\
&=&\dsum\limits_{0\leq \left\vert \beta \right\vert \leq k}\parallel
D^{\beta }f\parallel _{\tciLaplace ^{\psi }\left( \Omega ,Cl_{n}\right) } \\
&\leq &\gamma \parallel f\parallel _{W^{k,\psi }\left( \Omega ,Cl_{n}\right)
}
\end{eqnarray*}
\end{proof}

\begin{proposition}
We also have the mapping properties of the boundary Feuter integral $\xi
_{\partial \Omega }$ and the trace operator $\tau $:

$(i)$ The Feuter transform :

\begin{equation}
\xi _{\partial \Omega }:\widetilde{W}^{k-1,\psi ,\psi }\left( \partial
\Omega ,Cl_{n}\right) \longrightarrow W^{k,\psi }\left( \Omega ,Cl_{n}\right)
\label{Feuter property}
\end{equation}%
with 
\begin{equation*}
\Vert \xi _{\partial \Omega }f\Vert _{W^{k,\psi }\left( \Omega
,Cl_{n}\right) }\leq \theta \Vert f\Vert _{\widetilde{W}^{k-1,\psi ,\psi
}\left( \partial \Omega ,Cl_{n}\right) }
\end{equation*}

and

$(ii)$ the trace operator : 
\begin{equation}
\tau :W^{k,\psi }(\Omega ,Cl_{n})\longrightarrow \widetilde{W}^{k-1,\psi
,\psi }\left( \partial \Omega ,Cl_{n}\right)  \label{trace property}
\end{equation}%
with

\begin{eqnarray*}
\parallel \tau f\parallel _{\widetilde{W}^{k-1,\psi ,\psi }\left( \partial
\Omega ,Cl_{n}\right) } &&=\text{ }\parallel \tau f\parallel _{W^{k-1,\psi
}\left( \partial \Omega ,Cl_{n}\right) } \\
&&+\dsum\limits_{\left\vert \alpha \right\vert =k-1}\dint\limits_{\partial
\Omega }\dint\limits_{\partial \Omega }\psi \left( \frac{\left\vert
D^{\alpha }\tau f(x)-D^{\alpha }\tau f(y)\right\vert }{\lambda \left\vert
x-y\right\vert }\left\vert x-y\right\vert ^{2-n}d\partial \Omega
_{x}d\partial \Omega _{y}\right) \\
&& \\
&\leq &\theta _{1}\parallel f\parallel _{W^{k,\psi }\left( \Omega
,Cl_{n}\right) } \\
&&+\theta _{2}\dsum\limits_{\left\vert \alpha \right\vert
=k-1}\dint\limits_{\Omega }\dint\limits_{\Omega }\psi \left( \frac{%
\left\vert D^{\alpha }f(x)-D^{\alpha }f(y)\right\vert }{\lambda \left\vert
x-y\right\vert }\left\vert x-y\right\vert ^{2-n}d\partial \Omega
_{x}d\partial \Omega _{y}\right) \\
&& \\
&=&\theta \left( 
\begin{array}{c}
\parallel f\parallel _{W^{k,\psi }\left( \Omega ,Cl_{n}\right) } \\ 
+\dsum\limits_{\left\vert \alpha \right\vert =k-1}\dint\limits_{\Omega
}\dint\limits_{\Omega }\psi \left( \frac{\left\vert D^{\alpha
}f(x)-D^{\alpha }f(y)\right\vert }{\lambda \left\vert x-y\right\vert }%
\left\vert x-y\right\vert ^{2-n}d\partial \Omega _{x}d\partial \Omega
_{y}\right)%
\end{array}%
\right) \\
&& \\
&=&\theta \parallel f\parallel _{W^{k,\psi }(\Omega ,Cl_{n})}
\end{eqnarray*}

where $\theta _{1},\theta _{2}$ are quantities of $\left( n,\psi ,\Omega
\right) $ and $\delta =\delta \left( n,\psi ,\Omega \right) $ with $\theta
=\max \{\theta _{1},\theta _{2}\}$
\end{proposition}

\ \ \ \ \ \ \ \ \ \ \ \ \ \ \ \ 

\begin{proposition}
The composition $\xi _{\partial \Omega }\circ \tau $ preserves regularity of
a function in a Sobolev space.
\end{proposition}

\ \ \ \ \ \ \ \ \ \ \ \ \ \ \ \ \ \ \ \ 

\begin{proof}
Indeed the trace operator $\tau $ makes a function to loose a regularity
exponent of \ $one$ when acted along the boundary of the domain keeping
integrability index unchanged. But the boundary or \textit{Feuter} integral $%
\xi _{\partial \Omega }$ augments the regularity exponent of a function
defined on the boundary by an exponent that is lost by the trace operator
and therefore the composition operator $\xi _{\partial \Omega }\circ \tau $
preserves or restores the regularity exponent of a function in a Sobolev
space.

\ \ \ \ \ \ \ \ \ \ \ \ \ \ \ \ \ \ \ \ \ \ \ \ \ \ \ \ \ \ \ \ \ \ \ \ \ \
\ \ \ \ \ \ 
\end{proof}

\ \ \ \ \ \ \ \ \ \ \ \ \ \ \ \ \ \ 

The following proposition is what I call it the trinity of Clifford analysis
based on the relationship that connects $I,\xi _{\partial \Omega }$ and $%
\zeta _{\Omega }$ where $I$ is the identity operator.

\ 

\begin{proposition}
(Borel-Pompeiu ) Let $\ f\in W^{k,\psi }\left( \Omega ,Cl_{n}\right) .$ Then

\begin{equation*}
f=\xi _{\partial \Omega }\tau f+\zeta _{\Omega }Df
\end{equation*}
\end{proposition}

\begin{proof}
The proof can be done either through Gauss theorem or integration by parts
shown below first for a function $f\in C^{\infty }\left( \Omega
,Cl_{n}\right) \cap W^{k,\psi }\left( \Omega ,Cl_{n}\right) $ 
\begin{equation*}
\dint\limits_{\Omega }\Phi \left( x-y\right) Df(y)d\Omega
_{y}=\dint\limits_{\partial \Omega }\Phi \left( x-y\right) n(y)f(y)d\partial
\Omega _{y}-\dint\limits_{\Omega }D\Phi \left( x-y\right) f(y)d\Omega _{y}
\end{equation*}%
But 
\begin{equation*}
\dint\limits_{\Omega }D\Phi \left( x-y\right) f(y)d\Omega
_{y}=\dint\limits_{\Omega }\delta \left( x-y\right) f(y)d\Omega _{y}=f(x)
\end{equation*}

where $\delta $ here is the Dirac-delta (impulse) distribution and
rearranging terms we get the result.

Then since $C^{\infty }\left( \Omega ,Cl_{n}\right) \cap W^{k,\psi }\left(
\Omega ,Cl_{n}\right) $ is dense in $W^{k,\psi }\left( \Omega ,Cl_{n}\right) 
$ and by continuity arguments for $f\in W^{k,\psi }\left( \Omega
,Cl_{n}\right) $ we get a sequence $\{f_{n}:n\in 
\mathbb{N}
\}\subseteq C^{\infty }\left( \Omega ,Cl_{n}\right) \cap W^{k,\psi }\left(
\Omega ,Cl_{n}\right) $ such that $f_{n}\longrightarrow f$ in $W^{k,\psi
}\left( \Omega ,Cl_{n}\right) $ sense and that completes the proof.
\end{proof}

\begin{corollary}
$(i)$ \ If $f\in W_{0}^{k,\psi }\left( \Omega ,Cl_{n}\right) $, then 
\begin{eqnarray*}
f\left( x\right) &=&\dint\limits_{\Omega }\Phi \left( x-y\right)
Df(y)d\Omega _{y} \\
&=&\xi _{\Omega }Df
\end{eqnarray*}%
That is $D$ is a right inverse for $\zeta _{\Omega }$ and $\zeta _{\Omega }$
is a left inverse for $D$ over traceless spaces.

$(ii)$ \ If $f$ is monogenic function over $\Omega $, then%
\begin{eqnarray*}
f\left( x\right) &=&\dint\limits_{\partial \Omega }\Phi \left( x-y\right)
n(y)f(y)d\partial \Omega _{y} \\
&=&\xi _{\partial \Omega }\tau f
\end{eqnarray*}

Therefore monogenic functions are always Cauchy transforms of their traces
over the boundary.
\end{corollary}

\begin{proof}
The proof follows from the above Borel-Pompeiu result. But a further note
from $(i)$ and $(ii)$ of the corollary is that a traceless monogenic
function is a null function.
\end{proof}

\section{$\protect\bigskip \mathbf{Decomposition\ Results}$}

In this section we present two decomposition results, one for the $Cl_{n}-$
valued Orlicz space $\tciLaplace ^{\psi }\left( \Omega ,Cl_{n}\right) $ and
for the generalized Orlicz-Sobolev space $W^{k,\psi }\left( \Omega
,Cl_{n}\right) $.

But first,

\begin{definition}
Let $\psi $ be an Orlicz function, we define

$\left( i\right) $ \ The $\psi -$ Orlicz - Bergman space 
\begin{equation*}
A^{\psi }\left( \Omega ,Cl_{n}\right) :=\{f\in \tciLaplace ^{\psi }\left(
\Omega \longrightarrow Cl_{n}\right) :Df=0\text{ on }\Omega \}=\tciLaplace
^{\psi }\left( \Omega ,Cl_{n}\right) \cap \ker D
\end{equation*}
and

$\left( ii\right) $ $\ $The generalized $\psi -$ Orlicz-Sobolev - Bergamn
space 
\begin{equation*}
A^{k,\psi }\left( \Omega ,Cl_{n}\right) :=W^{k,\psi }\left( \Omega
,Cl_{n}\right) \cap \ker D
\end{equation*}
\end{definition}

\ The first decomposition result for the Orlicz-Sobolev space:

\begin{proposition}
Let $\psi :[0,\infty )\longrightarrow \lbrack 0,\infty )$ be an Orlicz
function. Then we have the direct decomposition of the Orlicz space 
\begin{equation*}
\tciLaplace ^{\psi }\left( \Omega ,Cl_{n}\right) =A^{\psi }\left( \Omega
,Cl_{n}\right) \oplus \overline{D}\left( W_{0}^{1,\psi }\left( \Omega
,Cl_{n}\right) \right)
\end{equation*}%
where $A^{\psi }\left( \Omega ,Cl_{n}\right) $ is the $\psi -$Orlicz -
Bergman space over $\Omega $.
\end{proposition}

\begin{proof}
Let%
\begin{equation*}
f\in A^{\psi }\left( \Omega ,Cl_{n}\right) \oplus \overline{D}\left(
W_{0}^{1,\psi }\left( \Omega ,Cl_{n}\right) \right)
\end{equation*}%
Then $Df=0$ and $f=\overline{D}g$ for some $g\in W_{0}^{1,\psi }\left(
\Omega ,Cl_{n}\right) $. But then 
\begin{equation*}
Df=D(\overline{D}g)=\Delta g_{\mid _{W_{0}^{1,\psi }\left( \Omega
,Cl_{n}\right) }}=\Delta _{0}g=0
\end{equation*}%
and from invertibility of $\Delta _{0}:$ $W_{0}^{1,\psi }\left( \Omega
,Cl_{n}\right) \longrightarrow Cl_{n}$, we see that $g=0$. Therefore $%
f\equiv 0$ which implies%
\begin{equation*}
A^{\psi }\left( \Omega ,Cl_{n}\right) \oplus \overline{D}\left(
W_{0}^{1,\psi }\left( \Omega ,Cl_{n}\right) \right) =\{0\}
\end{equation*}%
Again to show that every element $f\in \tciLaplace ^{\psi }\left( \Omega
,Cl_{n}\right) $ is a sum of elements form the summand spaces $A^{\psi
}\left( \Omega ,Cl_{n}\right) $ and $\overline{D}\left( W_{0}^{1,\psi
}\left( \Omega ,Cl_{n}\right) \right) $.

Let $f\in \tciLaplace ^{\psi }\left( \Omega ,Cl_{n}\right) $ and take $\eta
=\Delta _{0}^{-1}Df\in W_{0}^{1,\psi }\left( \Omega ,Cl_{n}\right) $, define
a function $g:=f-\eta $. Then $Dg=D\left( f-\eta \right) =0$ which implies%
\begin{equation*}
g\in KerD\cap \tciLaplace ^{\psi }\left( \Omega ,Cl_{n}\right) =A^{\psi
}\left( \Omega ,Cl_{n}\right)
\end{equation*}%
Thus%
\begin{equation*}
f=g\dotplus \eta \in A^{\psi }\left( \Omega ,Cl_{n}\right) \oplus \overline{D%
}\left( W_{0}^{1,\psi }\left( \Omega ,Cl_{n}\right) \right)
\end{equation*}%
where $\dotplus $ is used for elemental direct sum and that proves the
proposition.
\end{proof}

The second decomposition result for the generalized Orlicz-Sobolev space:

\begin{proposition}
The Clifford valued Sobolev space $W^{k,\psi }\left( \Omega ,Cl_{n}\right) \ 
$\ has a similar direct decomposition

\begin{equation*}
W^{k,\psi }\left( \Omega ,Cl_{n}\right) =A^{k,\psi }\left( \Omega
,Cl_{n}\right) \dotplus \overline{D}\left( W_{0}^{k+1,\psi }\left( \Omega
,Cl_{n}\right) \right)
\end{equation*}%
where $A^{k,\psi }\left( \Omega ,Cl_{n}\right) $ is the generalized $\psi $-
Orlicz-Sobolev - Bergman space over $\Omega $.
\end{proposition}

\begin{proof}
The proof follows the same argument as above.\ \ \ \ \ \ \ \ \ \ \ \ \ \ \ \
\ \ \ \ \ \ \ \ \ \ \ \ \ 
\end{proof}

\section{$\mathbf{First\ Order\ EllipticBVP}$}

\ Here we look at first order elliptic boundary value problems of the Dirac
operator and provide norm estimates of a solution in terms of norms of the
input data.

\begin{proposition}
Let $\ f\in W^{k-1,\psi }\left( \Omega ,Cl_{n}\right) $ and $g\in \widetilde{%
W}^{k-1,\psi ,\psi }\left( \partial \Omega ,Cl_{n}\right) $\ for $k\geq 1$.
Then the first order elliptic BVP:

\begin{equation}
\left\{ 
\begin{array}{c}
Du=f\text{ \ in }\Omega \\ 
\tau u=g\text{ on }\partial \Omega%
\end{array}%
\right.  \label{BVP 1}
\end{equation}

has a solution $u\in W^{k,\psi }\left( \Omega ,Cl_{n}\right) $ given by 
\begin{equation*}
u\left( x\right) =\xi _{\partial \Omega }g+\zeta _{\Omega }f
\end{equation*}
\end{proposition}

\ \ \ \ \ \ \ \ \ 

\begin{proof}
The proof follows from the Borel-Pompeiu relation. As to where exactly $u$
and $g$ belong, we make the argument : $f$ \ is in $W^{k-1,\psi }\left(
\Omega ,Cl_{n}\right) $ and hence from the mapping property of $D$, we have $%
u$ to be a function in $W^{k,\psi }\left( \Omega ,Cl_{n}\right) $.

\ 

Also from the mapping property of the trace operator $\tau $ we have 
\begin{equation*}
\tau u=u|_{\partial \Omega }=g\in \widetilde{W}^{k-1,\psi ,\psi }\left(
\partial \Omega ,Cl_{n}\right)
\end{equation*}
\end{proof}

\begin{proposition}
The solution $u\in W^{k,\psi }\left( \Omega ,Cl_{n}\right) $ of the elliptic
BVP $\left( \ref{BVP 1}\right) $\ has a norm estimate :

\begin{eqnarray*}
\Vert u\Vert _{W^{k,\psi }\left( \Omega ,Cl_{n}\right) } &\leq &\gamma
_{1}\left( 
\begin{array}{c}
\underset{\Vert \alpha \Vert \leq k-1}{\sum }\int_{\partial \Omega }\psi
\left( \frac{\left\vert \left( D^{\alpha }g|\right) \right\vert }{\lambda }%
\right) d\partial \Omega _{x} \\ 
+\underset{\Vert \alpha \Vert =k-1}{\sum }\dint\limits_{\partial \Omega
}\dint\limits_{\partial \Omega }\psi \left( \frac{|D^{\alpha }g\left(
x\right) -D^{\alpha }g\left( y\right) |}{\lambda |x-y|}\right) \left\vert
x-y\right\vert ^{2-n}d\partial \Omega _{x}d\partial \Omega _{y}%
\end{array}%
\right) \\
&&+\gamma _{2}\left( \underset{\Vert \alpha \Vert =k-1}{\sum }\int_{\Omega
}\psi \left( \frac{|D^{\alpha }f(x)|}{\lambda }\right) d\Omega _{x}\right)
\end{eqnarray*}

where $\gamma _{1},\gamma _{2}$ are constants the depend on $p$,$n$ and $%
\Omega $.
\end{proposition}

\ \ \ \ \ \ \ \ \ \ \ \ \ \ \ \ \ \ \ \ \ \ \ \ \ \ \ \ \ \ \ 

\begin{proof}
Clearly from the mapping properties of $D$, $\zeta _{\Omega }$, $\tau $ and $%
\xi _{\partial \Omega }$ $D$ and because $g\in \widetilde{W}^{k-1,\psi ,\psi
}\left( \partial \Omega ,Cl_{n}\right) $ and $f\in W^{k-1,,\psi }\left(
\Omega ,Cl_{n}\right) $ we have%
\begin{equation*}
u\in W^{k,\psi }\left( \Omega ,Cl_{n}\right)
\end{equation*}%
From the Borel-Pompeiu theorem we have the solution $u$ given by: 
\begin{equation*}
u\left( x\right) =\xi _{\partial \Omega }g+\zeta _{\Omega }f
\end{equation*}%
Now the of the solution $u$ can be estimated in the following sequence of
inequalities:%
\begin{eqnarray*}
\Vert u\Vert _{W^{k,p}\left( \Omega ,Cl_{n}\right) } &=&\Vert \xi _{\partial
\Omega }g+\zeta _{\Omega }f\Vert _{W^{k,p}\left( \Omega ,Cl_{n}\right) } \\
&\leq &\Vert \xi _{\partial \Omega }g\Vert _{W^{k,p}\left( \Omega
,Cl_{n}\right) }+\Vert \zeta _{\Omega }f\Vert _{W^{k,p}\left( \Omega
,Cl_{n}\right) } \\
&\leq &\gamma _{1}\Vert g\Vert _{\widetilde{W}^{k-1,\psi ,\psi }\left(
\partial \Omega ,Cl_{n}\right) }+\gamma _{2}\Vert f\Vert _{W^{k-1,p}\left(
\Omega ,Cl_{n}\right) } \\
&& \\
&=&\gamma _{1}\left( 
\begin{array}{c}
\underset{\Vert \alpha \Vert \leq k-1}{\sum }\int_{\partial \Omega }\psi
\left( \frac{|D^{\alpha }g\left( x\right) |}{\lambda }\right) d\partial
\Omega _{x} \\ 
+\underset{\Vert \alpha \Vert =k-1}{\sum }\dint\limits_{\partial \Omega
}\dint\limits_{\partial \Omega }\psi \left( \frac{|D^{\alpha }g\left(
x\right) -D^{\alpha }g\left( y\right) |}{\lambda |x-y|}\right) \left\vert
x-y\right\vert ^{2-n}d\partial \Omega _{x}d\partial \Omega _{y}%
\end{array}%
\right) \\
&&+\gamma _{2}\left( \underset{\Vert \alpha \Vert =k-1}{\sum }\int_{\Omega
}\psi \left( \frac{|D^{\alpha }f\left( x\right) |}{\lambda }\right) d\Omega
_{x}\right) \\
&& \\
&=&\gamma _{1}\left( 
\begin{array}{c}
\underset{\Vert \alpha \Vert \leq k-1}{\sum }\int_{\partial \Omega }\psi
\left( \frac{|D^{\alpha }g\left( x\right) |}{\lambda }\right) d\partial
\Omega _{x} \\ 
+\underset{\Vert \alpha \Vert =k-1}{\sum }\dint\limits_{\partial \Omega
}\dint\limits_{\partial \Omega }\psi \left( \frac{|D^{\alpha }g\left(
x\right) -D^{\alpha }g\left( y\right) |}{\lambda |x-y|}\right) \left\vert
x-y\right\vert ^{2-n}d\partial \Omega _{x}d\partial \Omega _{y}%
\end{array}%
\right) \\
&&+\gamma _{2}\left( \underset{\Vert \alpha \Vert =k-1}{\sum }\int_{\Omega
}\psi \left( \frac{|D^{\alpha }f\left( x\right) |}{\lambda }\right) d\Omega
_{x}\right)
\end{eqnarray*}%
The constants $\gamma _{1}$ and $\gamma _{2}$ are from the mapping
properties of $\xi _{\partial \Omega },\zeta _{\Omega }$ and $\tau $.\ \ \ \
\ \ \ \ \ \ \ \ \ \ \ \ \ \ \ \ \ \ \ \ \ \ \ \ \ \ 
\end{proof}

\end{document}